\documentclass{article}
\usepackage[left=2cm,right=2cm,top=2cm,bottom=2cm, marginparwidth=4cm, marginparsep=0.5cm]{geometry} 
\usepackage{amsmath,amssymb,mathtools} 
\usepackage{amsthm}
\usepackage{graphicx}
\usepackage{mathrsfs}
\usepackage{marginnote} 
\usepackage[utf8]{inputenc}         
\usepackage[T1]{fontenc}            %
\usepackage[final]{microtype} 
\usepackage{url}                    
\usepackage{tikz-cd}                   
\usetikzlibrary{arrows}                 
\usepackage{adjustbox}
\usepackage[]{color} 
\usepackage{framed} 
\usepackage[hidelinks]{hyperref}

\usepackage{gauss}

\hypersetup{
  colorlinks   = true, 
  urlcolor     = blue, 
  linkcolor    = blue, 
  citecolor   = blue 
}

\usepackage[capitalise]{cleveref}

\tikzset{
commutative diagrams/.cd,
arrow style=tikz,
diagrams={>=latex}} 

\usepackage{titlesec} 
\titleformat*{\section}{\centering\large\sc }
\titleformat{\subsection}[runin]{\bfseries}{\thesubsection.}{3pt}{}

\usepackage{enumitem} 

\usepackage{tocloft} 

\theoremstyle{definition}
\newtheorem{thm}{Theorem}[section]
\newtheorem{defn}[thm]{Definition}
\newtheorem{rke}[thm]{Remark}
\newtheorem{ex}[thm]{Example}

\newtheorem{lm}[thm]{Lemma}

\newtheorem*{ththm}{Theorem}

\DeclareMathOperator{\supp}{Supp}
\DeclareMathOperator{\sing}{Sing}

\newtheoremstyle{break} 
  {\topsep}{\topsep}%
  {\itshape}{}%
  {\bfseries}{}%
  {\newline}{}%
  
\theoremstyle{break}

\newcommand{\rrr}{\mathbf{R}}
\newcommand{\qqq}{\mathbf{Q}}
\newcommand{\zzz}{\mathbf{Z}}
\newcommand{\ccc}{\mathbf{C}}
\newcommand{\nnn}{\mathbf{N}}

\def\no{\noindent}

\begin{document}

\title{Toric reduction of singularities for Newton nondegenerate $p$-forms}
\author{Bilal Balo}
\date{}
\maketitle

\begin{abstract}
    \no We study a class of holomorphic $p$-forms satisfying nondegeneracy conditions expressed through their Newton polyhedron and called Newton nondegenerate (NND). We give a characterization of NND $p$-forms by their toric reduction of singularities defined through a regular refinement of their dual fan. We then present an application of this result to the study of singularities of $(n-1)$-forms on $\ccc^n$.
\end{abstract}
\setcounter{tocdepth}{1}
\tableofcontents
\vspace{1cm}

\nocite{cox2024toric}

\section*{Introduction}

\noindent The study of the reduction of singularities of dimension 2 foliations dates back to the middle of the 19th century, where it was used in particular cases. It was developed further by the likes of Bendixson \cite{bendixson} and Dulac \cite{Dulac}, but the general algorithm was only established by Seidenberg in \cite{seidenberg1968reduction}. A simplified proof was published later by Arno van den Essen in \cite{van2006reduction}. Dumortier resolved this problem in the $\mathcal{C}^{\infty}$-case \cite{Dumortier}. The modern formulation of this theorem is as follows: let $\mathcal{F}$ be a foliation defined on the germ $(\ccc^2,0)$ by a $1$-form $\omega = a(x,y)dx+b(x,y)dy$. The origin is called a \emph{singular point} of $\mathcal{F}$ if $a(0,0)=b(0,0)=0$. We associate by duality the vector field $\chi=b(x,y)\frac{\partial}{\partial x}-a(x,y)\frac{\partial}{\partial y}$ to the form $\omega$. Suppose that the origin is a singular point and consider the  matrix of the linear part of $\chi$ at the origin:

\begin{equation*}
    \mathcal{J}_0(\chi,x,y) = \begin{pmatrix}
         \frac{\partial b}{\partial x}(0) & -\frac{\partial a}{\partial x}(0) \\
        \frac{\partial b}{\partial y}(0) & -\frac{\partial a}{\partial y}(0)
    \end{pmatrix}.
\end{equation*}

\noindent It is a \emph{nonnilpotent} singularity if the above matrix is nonnilpotent. If in addition, the quotient of its eigenvalues does not belong to $\qqq_{>0}$, is is called a \emph{reduced} singularity. This extra arithmetic condition is called \emph{nonresonance}.

\medskip

\noindent This definition is independent from the chosen coordinates $(x,y)$ and from the generator $\omega$ 
of the foliation $\mathcal{F}$. Seidenberg's theorem guarantees the existence of a finite sequence of blowing ups of points whose composition $\pi : (V,E) \longrightarrow (\ccc^2,0)$ is such that the strict transform of $\mathcal{F}$ by $\pi$ has at most reduced singularities in the exceptional divisor $E$.

\medskip 

\noindent For codimension 1 foliations in higher dimensional spaces, nonnilpotent and reduced singularities have been generalized by Cano (see \cite{cano1998reduction} for an introduction) respectively as \emph{presimple} and \emph{simple} singularities. The reduction theorem to simple points for codimension 1 foliations on threefolds was proven by Cano in 2004 \cite{cano2004reduction}.

\medskip

\noindent The codimension 2 case for threefolds corresponds to foliations defined by vector fields and was first treated by Daniel Panazzolo in the real analytic setting in \cite{10.1007/s11511-006-0011-7}, then by Panazzolo and McQuillan in the complex framework \cite{mcquillan2013almost}. Higher dimensional cases are still open, most of the available literature focuses on codimension 1 foliations. These theorems have applications to dynamical systems, birational geometry and Mori theory (\cite{brunella2015birational} and \cite{cascini2021mmp}).

\medskip

\no The Newton nondegeneracy condition (NND), first introduced in the context of function singularity theory (see \cite{zbMATH03550860}, \cite{zbMATH03520545} and \cite{zbMATH03546027}), ties successfully together toric geometry, Newton polyhedra and the study of singularities: any NND affine hypersurface admits a resolution of singularities by one toric morphism. Generalizations to higher codimensions of this result are established in \cite{aroca2013torical}, \cite{tevelev2007compactifications} and \cite{aroca2024groebner}.

\medskip

\no In the same spirit, we are concerned by the reduction of singularities of $p$-forms by the use of only one toric morphism. A version of Newton nondegeneracy has been formulated for plane vector fields by Brunella and Miari in \cite{brunella1990topological} to study topological classification problems. In \cite{isabel1999singular}, without assuming any Newton nondegeneracy condition, the authors characterize germs of foliations on $(\ccc^2,0)$ whose singularities become reduced after applying a toric morphism defined via their Newton polyhedron. In \cite{esterov2005indices}, Esterov defines the Newton polyhedron of a germ of $1$-form in $(\ccc^n,0)$ and introduces $\ccc$-genericity, a variant of the NND condition, together with toric resolution techniques to establish a formula for the index of a $1$-form on a complete intersection singularity. In dimension three, topological classification results for vector fields were proven in \cite{alonso2015infinitesimal}, assuming a different NND condition.

\medskip

\no More recently, Molina-Samper introduced a definition of NND codimension 1 foliations using Newton polyhedra systems (see \cite{molina2019combinatorial}, \cite{molina2020newton} and \cite{molina2022global}) and gave a characterization of such foliations in terms of the existence of a logarithmic reduction of singularities of combinatorial nature: in particular, it is a sequence of blowing ups with smooth centers of the ambient space. 

\medskip

\no Building upon this work, we would want to contribute to the study of NND holomorphic $p$-forms with the approach relying on toric geometry. This point of view was initiated by Brunella and Miari in \cite{brunella1990topological} for plane vector fields. Given a NND $p$-form, we will define explicitly a toric reduction of singularities through the dual fan associated to its Newton polyhedron.

\medskip

\no This article is motivated by works on embedded resolution of singularities with one toric morphism: \cite{teissier2003valuations}, \cite{teissier2014overweight}, \cite{Teissier2023}, \cite{mourtada2023jet}. While the general case of this conjecture remains open, it has been resolved for plane curve singularities (see \cite{mourtada2017jet}, \cite{de2023resolving}). We aim to formulate and study similar results, but in the context of foliations instead of algebraic varieties. A step in this direction is the statement $2)$ of \cref{Reduction theorem} below, where we give a characterization of NND $p$-forms via toric resolution of singularities.

\begin{ththm}
   Let $\eta$ be a nonidentically zero holomorphic $p$-form defined on $\ccc^n$. We denote by $\Gamma = \Gamma(\eta)$ its Newton polyhedron and by $\Sigma$ a regular refinement of the dual fan of $\eta$. Let $\pi : X_\Sigma \longrightarrow \ccc^n$ be the associated proper birational toric morphism. For every affine chart $\mathcal{U}_\sigma$ of $X_\Sigma$ defined by a maximal cone $\sigma\in \Sigma$ with minimal generators $v_1,\dots,v_n$, set $t_k=\min_{I\in \Gamma(\eta)}\langle v_k,I \rangle$, where $k\in\{1,\dots,n\}$ and $T=(t_1,\dots,t_n)$ and denote by $Y=(y_1,\dots,y_n)$  the coordinates which are canonically associated to the vectors $v_1, \dots, v_n$ on $\mathcal{U}_\sigma$. Set $A = \{k\in\{1,\dots,n\}, t_k > 0\}$.
   
    \begin{enumerate}
        \item[$1)$] The pull-back $\pi^{*}\eta$ can be expressed locally in every affine chart $\mathcal{U}_\sigma$ as 
        \[\pi^{*}\eta(Y) = Y^T\left( \sum_{|K| =  p}  \overline{f}_K(Y) \frac{dY_K}{Y_K}\right)\]
        where each coefficient $\overline{f}_K(Y)$ is holomorphic and divisible by the monomial $Y_{K\setminus A} = \prod_{k\in K\setminus A} y_k$.
        \item[$2)$] $\eta$ is Newton nondegenerate if and only if in every affine chart $\mathcal{U}_\sigma$, the coefficients $\overline{f}_K(Y)$ of the pull-back $\pi^{*}\eta$ have no common zeroes.
    \end{enumerate}
    
\end{ththm}

\no For $n=2$, Seidenberg's reduction gives reduced singularities, which are, in particular, nonnilpotent. This can be generalized to $(n-1)$-forms on $\ccc^n$ which are NND as in the following \cref{n-1}.

\begin{ththm}
   Let $\eta$ be a nonidentically zero holomorphic $p$-form on $\ccc^n$. If we suppose that $\eta$ is NND and $p=n-1$, then the strict transform of $\eta$ by the toric morphism defined via a regular refinement of its dual fan has at worst nonnilpotent singularities.
\end{ththm}

\no In the first part of this paper, we generalize Newton polyhedron of $1$-forms to $p$-forms and define associated Newton nondegeneracy conditions. The second part gathers results from toric geometry which are needed to formulate and prove our main result: statement $2)$ in \cref{Reduction theorem} which characterizes NND $p$-forms. The proof of the theorem and the application to $(n-1)$-forms are detailed in the last part. We don't require any integrability condition.

\vspace{1cm}

\no \textbf{Acknowledgements.} The author wishes to thank Hussein Mourtada, Matteo Ruggiero and Beatriz Molina-Samper for discussions and suggestions related to this work.

\section{Newton nondegeneracy condition}

In this paper we adopt the convention that the set $\nnn$ contains $0$:
$\nnn=\{0,1,2,3,\ldots\}$. Let $n$ and $p$ be fixed integers such that $n\ge 2$ and $1\le p<n$. 

\subsection{Newton polyhedron of a $p$-form.} We denote by $X=(x_1,\dots,x_n)$ the standard coordinate system on $\ccc^n$.  We identify $n$-tuples in $\ccc^n$ with row matrices. The notation $\langle \cdot | \cdot \rangle$ stands for the usual dot product on $\ccc^n$. Let $N$ be an $n$-dimensional smooth complex analytic space and $P\in N$, for any local coordinates $Z=(z_1,\dots,z_n)$ centered at $P$ and any $n\times n$ matrix $M=(v_{ij})$ with entries $v_{ij}$ in $\nnn$, we set $Z^M = (z_1^{v_{11}}\cdots z_n^{v_{1n}}, \dots,z_1^{v_{n1}}\cdots z_n^{v_{nn}})$. If $I=(i_1,\dots,i_n)\in \nnn^n$ is a line matrix, we will identify $Z^I$ with the monomial $z_1^{i_1}\dots z_n^{i_n}$. We will also use these notations for points $c=(c_1,\dots,c_n)\in \ccc^n$, namely: $c^M=(c_1^{v_{11}}\cdots c_n^{v_{1n}}, \dots,c_1^{v_{n1}}\cdots c_n^{v_{nn}})$ and $c^I =c_1^{i_1}\dots c_n^{i_n} $. For any finite set $J$, the number of elements of $J$ is denoted by $|J|$. When $J\subset\{1,\dots,n\}$ has $q$ elements $j_1<\dots<j_q$, we set $Z_J=z_{j_1}\dots z_{j_q}$ and $dZ_J = dz_{j_1}\wedge \dots \wedge dz_{j_q}$. If $q=p$, the logarithmic $p$-differentials are:

\begin{equation*}
    \frac{dZ_J}{Z_J} = \frac{dz_{j_1}}{z_{j_1}}\wedge \dots \wedge \frac{dz_{j_p}}{z_{j_p}}.
\end{equation*}

\medskip

\noindent Let $\eta$ be a holomorphic $p$-form defined on $\ccc^n$. Taking $N=\ccc^n$ and $Z=X$ allows us to express $\eta$ in a logarithmic way in the standard coordinate system:

\begin{equation}\label{logway}
\eta(X) = \sum_{|J|=p}f_J(X)\frac{dX_J}{X_J}.
\end{equation}

\smallskip

\no In the above expression, the subscript $|J| = p$ means that the sum is indexed by all the subsets $J\subset \{1,\dots,n\}$ with $|J| = p$. The coefficients $f_J(X)$ are holomorphic functions on $\ccc^n$ such that $\frac{f_J(X)}{X_J}$ is holomorphic. Expanding the functions $f_J(X)$ into power series, we obtain
\begin{equation}\label{powseries}
\eta(X) = \sum_{|J| = p} \left(\sum_{I\in \nnn^n} a_{I,J}X^{I}\right)\frac{dX_J}{X_J}, 
\end{equation}
where the coefficients $a_{I,J}$ are complex numbers and $X_J$ appears in every nonzero monomial $a_{I,J}X^{I}$. This means that if a coordinate of $I$ indexed by $J$ is zero, then $a_{I,J}=0$.  

\medskip

\noindent For any subset $A\subset \nnn^n$, we define $\eta_{A}$ by restricting the exponents in the expression of $\eta$ to $A$ :
\begin{equation*}
\eta_A(X) = \sum_{|J| = p} \left(f_J\right)_{|A}(X)\frac{dX_J}{X_J},
\end{equation*}

\noindent where 
\begin{equation*}
    \left(f_J\right)_{|A}(X)=\sum_{I\in A} a_{I,J}X^{I}.
\end{equation*}

\medskip

\noindent We then define the \emph{support} $\supp \eta$ of $\eta$: 
\begin{equation*}
    \supp \eta = \left\{I\in\nnn^n \mid \exists J\subset\{1,\dots,n\}, |J|=p \text{  and  }  a_{I,J}\ne 0\right\}.
\end{equation*}

\noindent \cref{Newtonpolyhedron} and \cref{nnd} below draw both from \cite{brunella1990topological}  and \cite{esterov2005indices}.

\begin{defn} \label[defn]{Newtonpolyhedron}
Let $B\subset \rrr_{\ge 0}^n$. The \emph{Newton polyhedron} of $B$, denoted by $\Gamma = \Gamma(B)$ is the convex envelope of 
\[\bigcup_{I\in B}\left(I+\rrr_{\ge 0}^n\right).\]
Its boundary is called the \emph{Newton boundary}. The dimension $1$ faces of $\Gamma$ are called \emph{edges} and the codimension $1$ one are \emph{facets}. The Newton polyhedron of $\eta$ is, by definition, the Newton polyhedron of its support and denoted by $\Gamma(\eta)$.
\end{defn}

\nocite{fernandez2022characterization} \nocite{loray:hal-00016434}

\begin{ex} \label[ex]{ex1.2}
    If $p=1$ and $n=2$, we denote by $X=(x,y)$ the standard coordinates on $\ccc^2$. Consider the $1$-form defined by $\eta(x,y) = (2x^2+xy)dx+(x^2+y^2+x)dy$. We express it logarithmically as in \cref{powseries}:
    \begin{equation*}
        \eta(x,y) = (2x^3+x^2y)\frac{dx}{x}+(x^2y+y^3+xy)\frac{dy}{y}
    \end{equation*}

\medskip

\noindent The support is $\supp \eta =\{(3,0),(2,1),(0,3),(1,1)\}$. The monomials $xydx$ and $x^2dy$ contribute to the same point $(2,1)$. In the figure below, the Newton polyhedron of $\eta$ and its boundary are represented. The latter is the union of four facets $F_1$, $F_2$, $F_3$ and $F_4$, of which $F_2$ and $F_3$ are the compact ones.

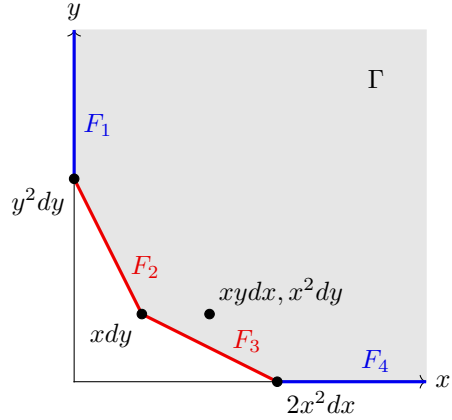
\begin{figure}[hbt!]
\centering

\begin{tikzpicture}[scale = 0.9]
	\draw[->] (-0.01,0) -- (5.2,0) node[right] {$x$};
	\draw[->] (0,-0.01) -- (0,5.2) node[above] {$y$};

    \draw[-, very thick, color = blue] (0.,3) -- (0,5.2); 
    \draw[color = blue, opacity = 0] (0,4.1) circle (2pt) node [anchor = north west, color = blue, opacity = 1] {$F_1$};
    \draw[domain=3:5.2, very thick, color = blue] plot (\x,{0}); 
    \draw[color = blue, opacity = 0] (4.1,0) circle (2pt) node [anchor = south west, color = blue, opacity = 1] {$F_4$};
	\draw[domain=0:1, very thick, color = red] plot (\x,{3-2*\x}); 
    \draw[color = red, opacity = 0] (0.7,2) circle (2pt) node [anchor = north west, color = red, opacity = 1] {$F_2$};
	\draw[domain=1:3, very thick, color = red] plot (\x,{3/2-1/2*\x});
    \draw[color = red, opacity = 0] (2.2,0.3) circle (2pt) node [anchor = south west, color = red, opacity = 1] {$F_3$};
    \draw[fill = black ,opacity = 0.1] (0,5.2)--(0,3)--(1,1)--(3,0)--(5.2,0)--(5.2,5.2)--(0,5.2);
    \draw[color = black, opacity = 0] (4.2,4.2) circle (2pt) node [anchor = south west, color = black, opacity = 1] {$\Gamma(\eta)$};

    \filldraw[black] (2,1) circle (2pt) node[anchor= south west] {$xydx,x^2dy$};
	\filldraw[black] (0,3) circle (2pt) node[anchor= north east] {$y^2dy$};
    \filldraw[black] (1,1) circle (2pt) node[anchor= north east] {$xdy$};
	\filldraw[black] (3,0) circle (2pt) node[anchor= north west] {$2x^2dx$};
	\end{tikzpicture}
    
\caption{Newton polyhedron of $\eta$.}\label{polnewton}
\end{figure}

\end{ex}

\begin{ex} \label[ex]{ex1.3}

If $p=2$ and $n=3$, let $X=(x,y,z)$ be the standard coordinate system on $\ccc^3$ and $\omega$ the $2$-form given by $\omega(x,y,z) = (z^6+xy)dx\wedge dy+xzdx\wedge dz+(x^6+x^4yz+yz)dy\wedge dz$. We write $\omega$ in a logarithmic way:
    \begin{equation*}
        \omega(x,y,z) = (xyz^6+x^2y^2)\frac{dx}x\wedge\frac{dy}y+x^2z^2 \frac{dx}x\wedge \frac{dz}z +(x^6yz+x^4y^2z^2+y^2z^2)\frac{dy}y\wedge \frac{dz}z
    \end{equation*}

\medskip

\noindent The support of the $2$-form $\omega$ is $\supp \omega =\{(1,1,6),(2,2,0),(2,0,2),(6,1,1),(4,2,2),(0,2,2)\}$. The Newton boundary of $\Gamma(\omega)$, represented in the next figure, is the union of a compact facet and six noncompact ones.

\bigskip

\begin{figure}[hbt!]
\begin{center}
\begin{tikzpicture}[x={(-5mm,-8mm)},z={(0,1cm)},y={(1cm,-.3cm)}, scale = 0.71]
\draw [->] (0,0) -- (7,0,0) node [right] {$x$};
\draw [->] (0,0) -- (0,7,0) node [above] {$y$};
\draw [->] (0,0) -- (0,0,7) node [below left] {$z$};
\draw (1,1,6) node [anchor = south west]{$z^6dx\wedge dy$};

\draw (6,1,1) node[anchor = north west]{$x^6 dy\wedge dz$};
\draw (0,2,2) node [anchor = south west]{$yz dy\wedge dz$};
\draw (2,2,0) node [anchor = north west]{$xydx\wedge dy$};
\draw (2,0,2) node [anchor = north east]{$xz dx\wedge dz$};

\draw[fill = blue, opacity = 0.4, shade] (0,2,2)--(2,0,2)--(2,0,7)--(0,2,7);
\draw[fill = blue,opacity = 0.4, shade] (2,0,2)--(2,2,0)--(7,2,0)--(7,0,2);
\draw[fill = blue,opacity = 0.4, shade] (2,0,2)--(2,0,7)--(7,0,7)--(7,0,2);
\draw[fill = red,opacity = 0.4, shade] (2,0,2)--(0,2,2)--(2,2,0);

\draw[fill = blue ,opacity = 0.4, shade] (2,2,0)--(0,2,2)--(0,7,2)--(2,7,0);
\draw[fill = blue ,opacity = 0.4, shade] (2,2,0)--(7,2,0)--(7,7,0)--(2,7,0);
\draw[fill = blue ,opacity = 0.4, shade] (0,2,2)--(0,2,7)--(0,7,7)--(0,7,2);

\draw[color = red, opacity = 0] (6,0,6) circle (2pt) node [anchor = north west, color = blue, opacity = 1] {$F_1$};
\draw[color = red, opacity = 0] (1,1,4.5) circle (2pt) node [anchor = north west, color = blue, opacity = 1] {$F_2$};
\draw[color = red, opacity = 0] (0,4.5,4.5) circle (2pt) node [anchor = north west, color = blue, opacity = 1] {$F_3$};
\draw[color = red, opacity = 0] (6,0,2) circle (2pt) node [anchor = north west, color = blue, opacity = 1] {$F_4$};
\draw[color = red, opacity = 0] (2.3,1.5,2.7) circle (2pt) node [anchor = north west, color = red, opacity = 1] {$F_5$};
\draw[color = red, opacity = 0] (0.4,4.2,1.6) circle (2pt) node [anchor = north west, color = blue, opacity = 1] {$F_6$};
\draw[color = red, opacity = 0] (4.5,4.5,0) circle (2pt) node [anchor = north west, color = blue, opacity = 1] {$F_7$};

to
\filldraw[black] (2,0,2) circle (2pt);
\filldraw[black] (2,2,0) circle (2pt);
\filldraw[black] (0,2,2) circle (2pt);
\filldraw[black] (6,1,1) circle (2pt);
\filldraw[black] (1,1,6) circle (2pt);
\filldraw[black] (4,2,2) circle (2pt);
\filldraw[black, opacity=0] (1.95,1.75,-0.4) circle (2pt) node [anchor = south east, opacity = 1]{$x^4yz dy \wedge dz$};

\end{tikzpicture}
\end{center}
\caption{Newton boundary of $\Gamma(\omega)$.}\label{boundary}
\end{figure}
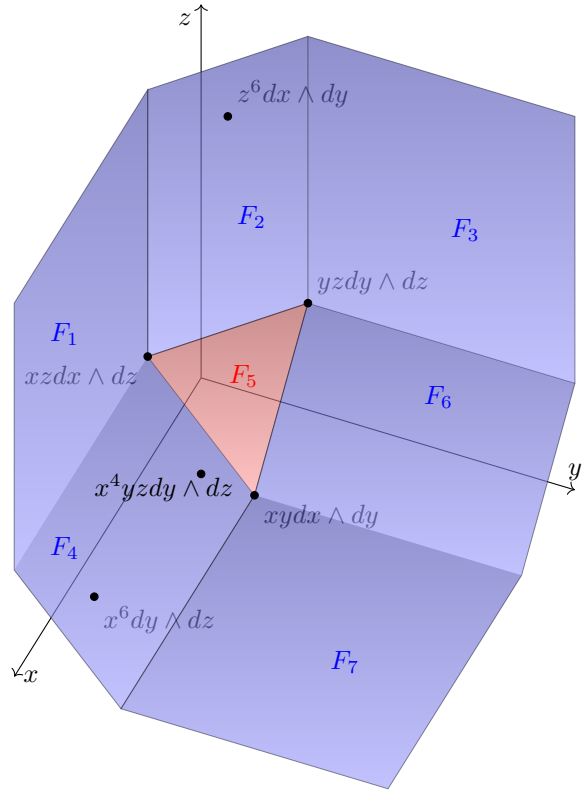
\end{ex}

\subsection{Newton nondegeneracy.}

\begin{defn}
    Let $\eta$ be a holomorphic $p$-form on $\ccc^n$ written as $\eta(X) =\sum_{|J|=p} g_J(X)dX_J$, where $g_J(X)$ are holomorphic functions. The \emph{singular locus} of $\eta$ is defined as the common zeroes of the coefficients $g_J$:
\begin{equation*}
    \sing \eta  = \left\{c\in \ccc^n \mid \forall J\subset\ \{1,\dots,n \}, |J| = p, \; g_J(c) =0  \right\}.
\end{equation*}

\no Its elements are called \emph{singularities} or \emph{singular points} of $\eta$.

\end{defn}

\begin{defn} \label[defn]{nnd}
Let $\eta$ be a holomorphic $p$-form on $\ccc^n$. We say that $\eta$ is \emph{Newton nondegenerate} (NND) if for any nonempty face $F\subset \Gamma(\eta)$, the so called \emph{initial form} $\eta_F$ has no singularities in $\left(\ccc^{*}\right)^n$. 

\begin{rke}
    Suppose $\eta $ is expressed logarithmically as in \cref{logway}:
    \begin{equation*}
        \eta(X) =\sum_{|J|=p} \frac{f_J(X)}{X_J} dX_J.
    \end{equation*}
    \no Let $F$ be a nonempty face of $\Gamma(\eta)$ and $J\subset\{1,\dots,n\}$ with $|J|=p$. As $\left(\frac{f_J}{X_J}\right)_{|F} = \frac{\left(f_J\right)_{|F}}{X_J}$, for any $c\in\left(\ccc^{*}\right)^n$, $\left(\frac{f_J}{X_J}\right)_{|F}(c) =0$ if and only if $\left(f_J\right)_{|F}(c) = 0$. Thus, \cref{nnd} is equivalent to the following: for any nonempty face $F\subset \Gamma(\eta)$ and $J\subset\{1,\dots,n\}$ with $|J|=p$, the functions $(f_J)_{|F}$ have no common zeroes in $\left(\ccc^{*}\right)^n$.
\end{rke}

\end{defn}

\begin{ex}
    We examine \cref{ex1.2} and \cref{ex1.3}, starting by $\eta(x,y) = (2x^3+x^2y)\frac{dx}{x}+(x^2y+y^3+xy)\frac{dy}{y}$.  

    \medskip
    
     \no A practical observation is that the NND condition is automatically verified for vertices, so they don't play a role in the definition. Nevertheless, let's check it this time by computing the initial forms: $\eta_{\{(0,3)\}}(x,y) = y^2 dy$, $\eta_{\{(1,1)\}}(x,y) = xdy$ and $\eta_{\{(3,0)\}}(x,y) = 2x^2dx$. As expected, they only have singularities in the coordinate hyperplanes. For the facets, the initial forms are $\eta_{F_1}(x,y) = y^2 dy$, $\eta_{F_2}(x,y) =(y^2+x)dy$, $\eta_{F_3}(x,y) = 2x^2dx+xdy$ and $\eta_{F_4}(x,y) = 2x^2dx$. We notice that $\eta_{F_2}$ has a curve of singularities intersecting $\left(\ccc^{*}\right)^2$. Let's look at the full face $F = \Gamma(\eta)$, the associated initial form $\eta_{\Gamma(\eta)}$ is $\eta$ itself. It has two singularities: $(0,0)$ and $(-\frac 15,\frac{2}{5})\in \left(\ccc^{*}\right)^2$. We conclude that the faces $F_2$ and $\Gamma(\eta)$ prevent $\eta$ from being NND.

    \bigskip

    \noindent We can find the initial forms for $\omega$ with \cref{boundary}. Notice that the support point $(4,2,2)$ given by the monomial $x^4yzdy\wedge dz$ does not belong to the Newton boundary, it is in the interior of the Newton polyhedron. 

    \smallskip
    
    \no To begin, we consider the edges of $\Gamma(\omega)$. For instance, $\omega_{F_1\cap F_2}(x,y,z) = xz dx\wedge dz$ is a monomial. This is also true for the initial forms defined by the remaining five noncompact edges. Let's examine the compact ones, for example: $\omega_{F_4 \cap F_5}(x,y,z) = xy dx \wedge dy + xzdx \wedge dz$. Here again the NND condition is fulfilled and this will also be the case for the $2$ other segments of $F_5$. Now, considering facets: $F_1$, $F_3$ and $F_7$ define monomials. The point $(4,2,2)$ associated to the monomial $x^4yzdy\wedge dz$ may be in the support of $\eta$, but it is outside of the face $F_4$, so $\omega_{F_4}(x,y,z) = xydx \wedge dy+ xz dx\wedge dz + x^6 dy\wedge dz$, whose singular locus is the hyperplane $\{x=0\}$. We can make the same type of verifications for $F_2$, $F_5$ and $F_6$. The initial form associated to the full face $\Gamma(\omega)$ is $\omega$. Its singular locus is defined by a polynomial system whose set of solutions is the $y$-axis $\{x=z=0\}$. We conclude that $\omega$ is NND.
\end{ex}

\section{Toric background} In this section, we collect some results on normal toric varieties. We follow closely the introduction found in \cite{aroca2013torical}. A general reference on toric geometry is \cite{cox2024toric}. One can also consult the second chapter in \cite{oka1997non} for an introduction to toric morphisms.

\subsection{Cones and fans}

    Let $v_1,\dots,v_q\in \rrr^n$. The \emph{convex polyhedral cone} generated by $v_1,\dots,v_q$ is the subset
\begin{equation*}
    \sigma = \left\{\lambda_1 v_1 +\dots +\lambda_q v_q, \lambda_k\in \rrr_{\ge 0}, k=1,\dots, q\right \}\subset \rrr^n.
\end{equation*}

\no We will often shortly refer to these objects as ‘cones’. The vectors $v_k$ are called \emph{generators} of the cone. A convex polyhedral cone is said to be \emph{rational} if it has a set of generators in $\zzz^n$. A rational polyhedral cone is said to be \emph{strongly convex} if it does not contain any nontrivial linear subspace.

\medskip

\no Let $\sigma \subset \rrr^n$ be a strongly convex rational polyhedral cone. The $1$-dimensional faces of $\sigma$ are half-lines called \emph{rays} and each of them has a unique generator in $\zzz^n$ whose coordinates have greatest common divisor 1. Considering all the $1$-dimensional faces of $\sigma$, we thus define a set of vectors in $\zzz^n$ which generates $\sigma$ (see Lemma 1.2.15. in \cite{cox2024toric} for more details). We call it the set of \emph{minimal generators} of the strongly convex rational polyhedral cone $\sigma$.

\medskip

\no A strongly convex rational polyhedral cone $\sigma \subset \rrr^n$ is \emph{smooth} if its minimal generators form part of a $\zzz$-basis of $\zzz^n$.

\medskip

\no A fan $\Sigma$ in $\rrr^n$ is a finite collection of cones $\sigma \subset \rrr^n$ such that:

\begin{enumerate}
    \item[$a)$] Each $\sigma\in \Sigma$ is a strongly convex rational polyhedral cone;
    \item[$b)$] For all $\sigma \in \Sigma$, each face of $\sigma$ is also in $\Sigma$;
    \item[$c)$] For all $\sigma_1,\sigma_2 \in \Sigma$, the intersection $\sigma_1 \cap \sigma_2$ is a face of each (hence, also in $\Sigma$).
\end{enumerate}

\no Given a fan $\Sigma$ in $\rrr^n$, we will say that $\sigma\in \Sigma$ is a \emph{maximal cone} of $\Sigma$ if it is maximal for the inclusion among the cones of $\Sigma$. The fan $\Sigma$ is \emph{regular} when each of its cones are smooth (in practice, it is sufficient to check that the maximal cones are smooth). Let $\Sigma$ and $\Sigma'$ be fans in $\rrr^n$, we say that $\Sigma'$ is a refinement of $\Sigma$ if every $\sigma \in \Sigma$ is a union of cones in $\Sigma'$. We refer to \cite{cox2024toric} for the proof of the existence of regular refinements.

\begin{thm}
    Any fan in $\rrr^n$ has a regular refinement.
\end{thm}

\begin{defn}
    Let $\eta$ be a nonidentically zero holomorphic $p$-form on $\ccc^n$. For any nonempty face $F\subset \Gamma(\eta)$, the normal vectors to the facets of $\Gamma(\eta)$ containing $F$ which are in the orthant $\rrr_{\ge 0}^n$ generate a cone denoted by $C_F$. The set of all the such cones
\begin{equation*}
    \Sigma(\eta) = \{C_F, \; F \text{ is a nonempty face of } \Gamma(\eta)\}
\end{equation*}

\no is a fan called the \emph{dual fan} of $\eta$.
\end{defn}

\subsection{Varieties and morphisms} \label{varsandmorphs}  We refer to \cite{cox2024toric} and \cite{oka1997non} for the proofs of the following results. For our purposes, we will only need the expression of a toric morphism in affine charts. 

\begin{defn}
    A \emph{toric variety} of dimension $n$ is an irreducible complex algebraic variety $V$ containing a torus $T_N\cong \left(\ccc^{*}\right)^n$ as a Zariski open subset such that the action of $T_N$ on itself extends to an action $T_N\times V \longrightarrow V$ given by a morphism of algebraic varieties.
\end{defn}

\no To any fan $\Sigma$ in $\rrr^n$, we can associate a normal separated toric variety $X_{\Sigma}$ (Theorem 3.1.5 in \cite{cox2024toric}). Moreover, if the fan $\Sigma$ is regular, then $X_{\Sigma}$ is a smooth variety.

\begin{defn}
    Let $\Sigma$ and $\Sigma'$ be fans in $\rrr^n$. A morphism $\pi : X_{\Sigma} \longrightarrow X_{\Sigma'}$ is \emph{toric} if $\pi$ maps the torus $T\subset X_{\Sigma}$ into the torus $T'\subset X_{\Sigma'}$ and $\pi_{|T}$ is a group morphism.
\end{defn}

\begin{thm} \label[thm]{refinement}
    Let $\Sigma$ and $\Sigma'$ be fans in $\rrr^n$. If $\Sigma$ is a refinement of $\Sigma'$, then it induces a proper birational toric morphism $\pi : X_{\Sigma} \longrightarrow X_{\Sigma'}$.
\end{thm}

\no Suppose that in the previous theorem, $\Sigma'$ is the fan whose only maximal cone is the orthant $\rrr_{\ge 0}^n$, then $X_{\Sigma'} = \ccc^n$. Let $\mathcal{U}_{\sigma}$ be an affine chart of $X_{\Sigma}$ defined by a maximal cone $\sigma$ with minimal generators $v_1,\dots,v_n$. For any $k\in\{1,\dots,n\}$, we can write $v_k = (v_{1k},\dots,v_{nk})\in \nnn^n$. Let $M$ be the matrix whose columns are given by the coordinates of the generators:
    \begin{equation*}
          M  = \bordermatrix{
           & v_{1} & v_{2} & \dots & v_{n} \cr
     & v_{11} & v_{12} & \dots & v_{1n} \cr
     & v_{21} & v_{22} &  \dots & v_{2n} \cr
     & \vdots & \ddots & \vdots & \vdots \cr
     & v_{n1} & v_{n2} & \dots & v_{nn} \cr
  }
      \end{equation*}

\smallskip

\no Let $Y=(y_1,\dots,y_n)$  the coordinates which are canonically associated to the vectors $v_1, \dots, v_n$ on $\mathcal{U}_\sigma$, then the morphism $\pi_{|\mathcal{U}_\sigma}$ is given by $X=Y^M$.

\section{Reduction of singularities}

\nocite{cano2013theorie}  \nocite{duqueelim} \nocite{cano2022truncated} \nocite{camacho1984topological}

\subsection{Toric reduction for NND $p$-forms}

\begin{thm} \label[thm]{Reduction theorem}
Let $\eta$ be a nonidentically zero holomorphic $p$-form defined on $\ccc^n$. We denote by $\Gamma = \Gamma(\eta)$ its Newton polyhedron and by $\Sigma$ a regular refinement of the dual fan of $\eta$. Let $\pi : X_\Sigma \longrightarrow \ccc^n$ be the associated proper birational toric morphism. For every affine chart $\mathcal{U}_\sigma$ of $X_\Sigma$ defined by a maximal cone $\sigma\in \Sigma$ with minimal generators $v_1,\dots,v_n$, set $t_k=\min_{I\in \Gamma(\eta)}\langle v_k,I \rangle$, where $k\in\{1,\dots,n\}$ and $T=(t_1,\dots,t_n)$ and denote by $Y=(y_1,\dots,y_n)$  the coordinates which are canonically associated to the vectors $v_1, \dots, v_n$ on $\mathcal{U}_\sigma$. Set $A = \{k\in\{1,\dots,n\}, t_k > 0\}$.
   
\begin{enumerate}
        \item[$1)$] The pull-back $\pi^{*}\eta$ can be expressed locally in every affine chart $\mathcal{U}_\sigma$ as 
        \[\pi^{*}\eta(Y) = Y^T\left( \sum_{|K| =  p}  \overline{f}_K(Y) \frac{dY_K}{Y_K}\right)\]
        where each coefficient $\overline{f}_K(Y)$ is holomorphic and divisible by the monomial $Y_{K\setminus A} = \prod_{k\in K\setminus A} y_k$.
        \item[$2)$] $\eta$ is Newton nondegenerate if and only if in every affine chart $\mathcal{U}_\sigma$, the coefficients $\overline{f}_K(Y)$ of the pull-back $\pi^{*}\eta$ have no common zeroes.
\end{enumerate}
\end{thm}

\begin{proof}

     \no  We will keep the same notations as in \cref{varsandmorphs}. We want to compute the pull-back of $\eta$ by $\pi$ in the chart $\mathcal{U}_{\sigma}$. The toric morphism $\pi$ is expressed in this chart by the change of variables $X = Y^{M}$ :
    \begin{equation}\label{chvar}
      \begin{cases}
  x_1 = y_1^{v_{11}}\cdots y_n^{v_{1n}}\\      
 x_2 = y_1^{v_{21}}\cdots y_n^{v_{2n}}\\
 \;\;\;\;\;\vdots \\
  x_n = y_1^{v_{n1}}\cdots y_n^{v_{nn}}
\end{cases}
\end{equation}

    \smallskip
      
   \no   For all $j\in\{1,\dots,n\}$, taking logarithmic differentials in \cref{chvar}, we obtain a change of base formula:
\begin{equation}\label{chbase1}
    \frac{dx_j}{x_j} = \sum_{k=1}^n v_{jk} \frac{dy_k}{y_k}.
\end{equation}

\noindent Let $J$ be a subset of $\{1,\dots,n\}$ with $p$ elements $j_1<\dots<j_p$, using \cref{chbase1}, we find

\begin{equation*}
    \frac{dX_J}{X_J} = \bigwedge_{s=1}^{p} \left(\sum_{k=1}^n v_{j_sk} \frac{dy_k}{y_k}\right),
\end{equation*}

\smallskip

\noindent from which we derive a change of base formula for $p$-logarithmic differentials:

\begin{equation}\label{chbase2}
    \frac{dX_J}{X_J}=\sum_{|K| =  p} |M_{J,K}| \frac{dY_K}{Y_K},
\end{equation}

\smallskip

\noindent where $|M_{J,K}|$ is the determinant of the submatrix of $M$ whose rows and columns are indexed respectively by $J$ and $K$. Recall that $\eta$ is expressed in the standard coordinate system $X$ as in \cref{logway}:

\begin{equation*}
\eta(X) = \sum_{|J|=p}f_J(X)\frac{dX_J}{X_J}.
\end{equation*}

\smallskip

\noindent Combining \cref{chbase2} with the previous one, we start to compute the total transform of $\eta$ by $\pi$:

\begin{equation*}
    \begin{gathered}
        \pi^{*}\eta(Y) = \sum_{|J|=p}f_J(Y^M)\sum_{|K| =  p} |M_{J,K}| \frac{dY_K}{Y_K} \\
        \pi^{*}\eta(Y) = \sum_{|K|=p}\left(\sum_{|J| =  p} |M_{J,K}| f_J(Y^M)\right) \frac{dY_K}{Y_K}.
    \end{gathered}
\end{equation*}

\smallskip

\noindent Let us expand the coefficients $f_J$ into power series, keeping the same notations as in \cref{powseries}:

\begin{equation*}
    \pi^{*}\eta(Y) = \sum_{|K|=p}\left(\sum_{|J| =  p} |M_{J,K}| \left(\sum_{I\in \nnn^n\cap \Gamma} a_{I,J}Y^{I\cdot M}\right)\right) \frac{dY_K}{Y_K},
\end{equation*}

\smallskip

\noindent where the dot $\cdot$ denotes the matrix product. 

\begin{equation*}
    \pi^{*}\eta(Y) = \sum_{I\in \nnn^n\cap \Gamma}\left(\sum_{|K|=  p}\sum_{|J|=p}  a_{I,J} |M_{J,K}| \frac{dY_K}{Y_K}\right) Y^{ I\cdot M}.
\end{equation*}

\medskip

\noindent For any $I\in \nnn^n\cap \Gamma$, recall that $Y^{ I\cdot M} = y_1^{\langle v_1,I \rangle}\dots y_n^{\langle v_n,I \rangle}$. For all $k\in\{1,\dots,n\}$, we have $t_k = \min_{I\in \Gamma} \langle v_k,I \rangle$ and we denote by $F_k$ the face of $\Gamma$ defined by $v_k$, namely $F_k = \{I\in \Gamma, \langle v_k, I \rangle = t_k\}$. Let $A = \{k\in\{1,\dots,n\}, t_k > 0\}$. We consider the divisor $\left(\prod_{k\in A}y_k= 0\right)$. This leads us to factor the previous expression by $Y^T$, where $T=(t_1,\dots,t_n)$:

\begin{equation*}
    \pi^{*}\eta(Y) = Y^T\Biggl(\sum_{|K|=p}
\Bigl(
\underbrace{
\sum_{I\in \nnn^n\cap \Gamma}\sum_{|J| =  p} a_{I,J} |M_{J,K}| Y^{I\cdot M - T}
}_{\overline{f}_K(Y)}
\Bigr)
\frac{dY_K}{Y_K}
\Biggr).
\end{equation*}

\medskip

\no For any $K\subset \{1,\dots,n\}$ such that $|K|=p$, we can rewrite $\overline{f}_K(Y)$ with dot products:

\begin{equation} \label{dotprod}
    \overline{f}_K(Y)=\sum_{I\in \nnn^n\cap \Gamma}\sum_{|J| =  p}  a_{I,J} |M_{J,K}| y_1^{ \langle v_1,I\rangle -t_1}\dots y_n^{ \langle v_n,I \rangle -t_n}.
\end{equation}

\smallskip

\no \textbullet \; \textbf{Proof of statement $1)$.}

\medskip

\no Let $K\subset \{1,\dots,n\}$ such that $|K|=p$ and $k\in K\setminus A$, we will show that $y_k$ divides $\overline{f}_K(Y)$. To do so, we substitute $y_k = 0$ in \cref{dotprod}:

\begin{equation*} 
    \overline{f}_K(Y)_{|y_k=0}=\sum_{I\in F_k}\sum_{|J| =  p}  a_{I,J} |M_{J,K}| \prod_{s\ne k} y_s^{ \langle v_s,I\rangle -t_s}.
\end{equation*}

\smallskip

\no Let $I\in F_k$ and $J\subset \{1,\dots,n\}$ with $|J|=p$. We will justify that the product $a_{I,J}|M_{J,K}|$ is equal to zero. This will show that the sum above is also zero and that $y_k$ divides $\overline{f}_K(Y)$. Recall that  $a_{I,J} = 0$ whenever a coordinate of $I$ indexed by $J$ is zero (this has been stated after \cref{powseries}). Suppose now that all coordinates of $I$ indexed by $J$ are nonzero and let us prove that $|M_{J,K}|=0$.

\medskip

\no As $k\notin A$, we have $t_k=\min_{I\in\Gamma} \langle v_k,I\rangle = 0$. This implies that $v_k$ has at least one coordinate which is $0$. Let $\Lambda=\{j\in\{1,\dots,n\}, v_{jk} >0\}$. The face $F_k=\{I\in \Gamma, \langle v_k,I \rangle = 0\}$ is the intersection of $\Gamma$ with coordinate hyperplanes, namely: $F_k=\{I=(i_1,\dots,i_n)\in \Gamma, \forall s\in \Lambda, i_s = 0\}$. The matrix $M$ can be written as:

\newenvironment{mypmatrix}{\def\mathstrut{\vphantom{\big(}}\pmatrix}{\endpmatrix}

\begin{equation*}
    M = \bordermatrix{     
            & v_{1}     & \cdots    & v_k   & \cdots & v_n \cr
         & v_{11}     & \ldots    & *  & \ldots  & v_{1n}  \cr
         & \vdots     & \ldots     & *  &  \ldots  & \vdots \cr
         & \vdots  & \ldots       & \vdots     & \ldots  & \vdots \cr
         & \vdots     & \ldots      & \vdots   & \ldots  & \vdots \cr
         & v_{n1}     & \ldots      & *   & \ldots  & v_{nn} \cr
}.
\end{equation*}

\medskip

\no The coefficient $|M_{J,K}|$ is the determinant of the matrix whose rows are indexed by $J$ and columns are indexed by $K$. As $k\in K$, the $k$-th column of $M$ will be part of this extracted matrix. The elements of $\Lambda$ do not belong to $J$. Indeed, if there was an element $s\in \Lambda\cap J$, as $F_k=\{I=(i_1,\dots,i_n)\in \Gamma, \forall s\in \Lambda, i_s = 0\}$ and $I\in F_k$, the $s$-th coordinate of $I$ would be indexed by an element of $J$ and equal to zero. This contradicts our hypothesis. In the end, we get that all nonzero elements of the $k$-th column are placed in rows indexed by $\{1,\dots, n\}\setminus J$. Thus, the submatrix that we consider has a column of zeroes and its determinant $|M_{J,K}|$ is zero.

\medskip

\no For any $k\in K\setminus A$, $y_k$ divides $\overline{f}_K(Y)$, so $Y_{K\setminus A}$ divides $\overline{f}_K(Y)$. 

\medskip

\no \textbullet \; \textbf{Proof of the direct implication in statement $2)$.} 

\medskip

\no Suppose that there exists a point $c=(c_1,\dots,c_n)\in \ccc^n$ which is a common zero for the coefficients $\overline{f}_K(Y)$. Let $S =\{k\in\{1,\dots,n\}, c_k\ = 0\}$. 

\medskip

\nocite{brondsted2012introduction}

\no Consider the nonempty face $F$ of $\Gamma$ defined by $F=\cap_{k\in S}F_k$. Since $c$ is a common zero of the coefficients $\overline{f}_K(Y)$, for each subset $K\subset \{1,\dots,n\}$, using \cref{dotprod}, we get:

\begin{equation*}
    \sum_{I\in F} \sum_{|J|=p}a_{I,J}|M_{J,K}| \prod_{k\notin S} c_k^{ \langle v_k,I\rangle -t_k} = 0 
\end{equation*}
\begin{equation} \label{cancellation}
    \sum_{|J|=p} \left(\sum_{I\in F} a_{I,J} \prod_{k\notin S} c_k^{ \langle v_k,I\rangle}\right)|M_{J,K}|  = 0.
\end{equation}

\medskip

\no We observe that for any $I\in\Gamma \cap\nnn^n$,

\begin{equation*}
    \prod_{k\notin S} c_k^{ \langle v_k,I\rangle}  = \prod_{k\notin S}c_k^{i_1 v_{1k}+\dots+i_nv_{nk}} = \prod_{j=1}^n\prod_{k\notin S} (c_k^{v_{jk}})^{i_j} = C^I.
\end{equation*}

\medskip

\no where $C=(\prod_{k\notin S}c_k^{v_{1k}},\dots,\prod_{k\notin S}c_k^{v_{nk}}) \in (\ccc^{*})^n$. This helps us further simplify \cref{cancellation}:

\begin{equation*}
    \forall K, \; \sum_{|J|=p} \left(\sum_{I\in F} a_{I,J} C^I\right)|M_{J,K}|  = 0
\end{equation*} 
\begin{equation} \label{eqzerobis}
    \sum_{|J|=p} (f_J)_{|F}(C)|M_{J,K}|  = 0
\end{equation}

\no Let $J_1 \prec \dots \prec J_{r}$ be an ordering of the subsets of $\{1,\dots,n\}$ with $p$ elements, where $r=\binom{n}{p}$.

\no We introduce the following matrix: 
\begin{equation*}
    M'=\begin{pmatrix}
    |M_{J_1,J_1}| & |M_{J_1,J_2}| & \dots & |M_{J_1,J_r}| \\
    |M_{J_2,J_1}| & |M_{J_2,J_2}| & \dots & |M_{J_2,J_r}| \\
    \vdots & \vdots & \ddots & \vdots \\
    |M_{J_r,J_1}| & |M_{J_r,J_2}| & \dots & |M_{J_r,J_r}|
 \end{pmatrix}.
\end{equation*}

\smallskip

\no We proceed to reformulate the $r=\binom{n}{p}$ equalities in \cref{eqzerobis} as a matrix equation:
\begin{equation} \label{matrix}
    \left((f_{J_1})_{|F}(C),\dots, (f_{J_r})_{|F}(C)\right) M' = 0.
\end{equation}

\no Since $\sigma$ is a smooth cone, $M$ is invertible. Let us prove that $M'$ is also invertible. Denote by $(e_1, \dots, e_n)$ the canonical basis of $\ccc^n$ and $u : \ccc^n \longrightarrow \ccc^n$ the $\ccc$-linear map represented by the matrix $M$ in this basis. For any subset $J\subset\{1,\dots,n\}$ such that $J = \{j_1,\dots, j_p\}$ and $j_1<\dots < j_p$, set $e_J = e_{j_1}\wedge \dots \wedge e_{j_p}$. The $p$-th exterior power $\wedge^{p}u$ of the map $u$ is defined by $\wedge^p u(e_J) = u(e_{j_1})\wedge \dots \wedge u(e_{j_p})$. For any subset $K\subset \{1,\dots,n\}$ with $p$ elements $k_1<\dots<k_p$,

\begin{equation*}
    \wedge^p u(e_{K}) = \bigwedge_{\ell = 1}^pu(e_{k_\ell}) = \bigwedge_{\ell = 1}^p\left(\sum_{i=1}^nv_{ik_{\ell}}e_i\right) = \sum_{|J|= p }|M_{J,K}| e_J.
\end{equation*}

\medskip

\no The family of vectors $(e_{J_1},\dots,e_{J_r})$ is a basis of $\wedge^p \ccc^n$. For any $\ell \in\{1,\dots, r\}$,

\begin{equation*}
    \wedge^p u(e_{J_\ell}) = \sum_{s=1}^r|M_{J_s,J_\ell}| e_{J_s}.
\end{equation*}

\medskip

\no This proves that $M'$ is the matrix of $\wedge^p u$ in the basis $(e_{J_1},\dots,e_{J_r})$. The exterior power of an invertible linear endomorphism is invertible. Thus, $\wedge^p u$ is invertible and so is the matrix $M'$.

\medskip

\no This implies that the unique solution of \cref{matrix} is
\begin{equation*}
    \left((f_{J_1})_{|F}(C),\dots, (f_{J_r})_{|F}(C)\right)=(0,\dots,0)
\end{equation*}

\smallskip

\no and that $C\in \left(\ccc^{*}\right)^n$ is a common zero for the initial forms $(f_{J_k})_{|F}$. This is impossible because $\eta$ is NND, so the coefficients $\overline{f}_K(Y)$ have no common zeroes.

\medskip

\no \textbullet \; \textbf{Proof of the inverse implication in statement $2)$.} 

\medskip

\no Let's suppose that $\eta$ is not NND and find an affine chart where the coefficients of the total transform have a common zero. By hypothesis, there exists a nonempty face $F\subset \Gamma(\eta)$ and a point $c=(c_1,\dots,c_n)\in \left(\ccc^{*}\right)^n$ where all the initial forms $(f_J)_{|F}$ vanish. 

\medskip
    
\no As $\Sigma$ is a regular refinement of the dual fan of $\Gamma(\eta)$, we can always find an affine chart which "captures" the face $F$. More precisely: there is a chart $\mathcal{U}_\sigma$ defined by a maximal cone $\sigma\in \Sigma$ with minimal generators $v_1,\dots, v_n$ such that some vectors $v_k$ define $F$: 

\[F = \bigcap_{k\in R} F_k\]

\no where $R\subset \{1,\dots,n\}$. Set $c' = c^{M^{-1}}$. For any $k\in \{1,\dots,n\}$, set $\tilde{c}_k =0$ if $k\in R$ and $\tilde{c}_k= c'_k$ if $k\notin R$. Let $\tilde{c} = (\tilde{c}_1,\dots,\tilde{c}_n)$. 

\medskip

\no Using the expression of the coefficients $\overline{f}_K$ in \cref{dotprod}, we evaluate at $\tilde{c}$:
\begin{gather*}
    \forall K, \; \overline{f}_K(\tilde{c}) = \sum_{I\in \nnn^n\cap \Gamma}\sum_{|J| =  p} a_{I,J} |M_{J,K}|\tilde{c}_1^{ \langle v_1,I\rangle -t_1}\dots \tilde{c}_n^{ \langle v_n,I \rangle -t_n} \\ 
    \overline{f}_K(\tilde{c}) = \sum_{I\in F}\sum_{|J| =  p} a_{I,J} |M_{J,K}|\prod_{k\notin R} \tilde{c}^{\langle v_k,I \rangle -t_k}_k \\
    \overline{f}_K(\tilde{c}) = \sum_{|J| =  p} |M_{J,K}|\sum_{I\in F} a_{I,J} \prod_{k\notin R} c^{\prime \langle v_k,I \rangle-t_k}_k
\end{gather*}

\no For any $I\in F$ and $k\in R$, $c^{\prime \langle v_k,I \rangle-t_k}_k = (c'_k)^0 = 1$, hence we have $\prod_{k\notin R} c^{\prime \langle v_k,I \rangle-t_k}_k = \prod_{k= 1} ^n c^{\prime \langle v_k,I \rangle-t_k}_k $.
\begin{gather*}
    \overline{f}_K(\tilde{c}) = \sum_{|J| =  p} |M_{J,K}|\sum_{I\in F} a_{I,J} \prod_{k= 1} ^n c^{\prime \langle v_k,I \rangle-t_k}_k \\ 
    \overline{f}_K(\tilde{c}) = \sum_{|J| =  p} |M_{J,K}|\sum_{I\in F} a_{I,J} \left(c^{\prime}\right)^{I \cdot M - T} \\
    \overline{f}_K(\tilde{c}) = \sum_{|J| =  p} |M_{J,K}|\sum_{I\in F} a_{I,J} c^{I \cdot M \cdot M^{-1} - TM^{-1}}
    \end{gather*}

\smallskip

\no because $c' = c^{M^{-1}}$. Factoring by $c^{ - TM^{-1}}$, we find
    
\begin{gather*}
    \overline{f}_K(\tilde{c}) = c^{ - TM^{-1}} \sum_{|J| =  p} |M_{J,K}|\sum_{I\in F} a_{I,J} c^{I} \\
    \overline{f}_K(\tilde{c}) = c^{ - TM^{-1}} \sum_{|J| =  p} |M_{J,K}| (f_J)_{|F}(c) \\
    \overline{f}_K(\tilde{c}) = 0
\end{gather*}

\no since $(f_J)_{|F}(c) = 0$ by hypothesis.

\end{proof}

\subsection{An application to $(n-1)$-forms} 

\begin{defn}
    To any holomorphic $(n-1)$-form $\eta$ on $\ccc^n$ expressed in the standard coordinates by

\[\eta(X) = \sum_{|J|=n-1}g_J(X)dX_J,\]

\no we can associate by duality a vector field $\chi$ such that:

\[\chi(X) = \sum_{k=1}^n(-1)^{k-1}g_{J_k}(X)\frac{\partial}{\partial x_k},\]

\no where $J_k=\{1,\dots,n\}\setminus\{k\}$ for all $k=1,\dots,n$. A singular point $c$ of $\eta$ is said to be \emph{nonnilpotent} if the matrix of the linear part of $\chi$ at $c$ is nonnilpotent.
\end{defn}

\begin{thm} \label[thm]{n-1}
    Let $\eta$ be a nonidentically zero holomorphic $p$-form on $\ccc^n$. If we suppose that $\eta$ is NND and $p=n-1$, then the strict transform of $\eta$ by the toric morphism defined via a regular refinement of its dual fan has at worst nonnilpotent singularities.
\end{thm}

\begin{proof}

We will use the same notations as in \cref{Reduction theorem} and its proof. We first establish a general lemma.
Recall that the pull-back $\pi^{*}\eta$ can be expressed locally in an affine chart $\mathcal{U}_\sigma$ as 
        \[\pi^{*}\eta(Y) = Y^T\left( \sum_{|K| =  p}  \overline{f}_K(Y) \frac{dY_K}{Y_K}\right).\]

\begin{lm}
    The set $A=\{k\in \{1,\dots,n\},  t_k >0\}$ has at least $p$ elements.
\end{lm}

\no Let $B = \{1,\dots,n\}\setminus A$ and consider the face $G = \cap_{k\in B} F_k\subset \Gamma(\eta)$. As $G$ is nonempty, there exists $I'\in G$ and by definition, we have:
\begin{equation}
    \forall k\in B, \; \langle v_k,I' \rangle = 0. \label{subspace}
\end{equation}

\no The Newton polyhedron $\Gamma(\eta)$ is defined using the series expansion of the coefficients $f_J$ from the logarithmic expression of $\eta$. As explained below \cref{logway}, for any $J$, the monomial $X_J$ divides $f_J$. Thus, for each nonzero $a_{I,J}X^{I}$ in the series expansion of $f_J$, the $n$-tuple $I$ will have at least $p$ nonzero coordinates. We apply this to $I=I'$. Let $k\in B$, the equalities in \cref{subspace} imply that for each nonzero coordinate of $I'$, the coordinate of $v_k$ in the same position is zero. As $(v_k)_{k\in B}$ is a free family of vectors in an $(n-p)$-dimensional subspace of $\rrr^n$, we conclude that $|B|\le n-p$ and $|A|\ge p$. 

\medskip

\no \textbullet \; \textbf{Proof of Theorem 3.3} 

\medskip

\no Let us suppose that $\eta$ is NND and $p=n-1$. According to the previous lemma, $|A|\ge n-1$. If $|A|=n$, then the strict transform $\eta'$ of $\eta$ by $\pi$ can be expressed as:

\begin{gather*}
    \eta'(Y) = y_1\dots y_n\left( \sum_{|K| =  n-1}  \overline{f}_K(Y) \frac{dY_K}{Y_K}\right) \\
    \eta'(Y) = \sum_{k=1}^n y_k\overline{f}_{J_k}(Y)dY_{J_k}.
\end{gather*}

with $J_k = \{1,\dots,n\}\setminus\{k\}$ for all $k=1,\dots,n$. For any $k$, we set $h_k = \overline{f}_{J_k}$. Consider the vector field $\chi$ such that

\[\chi(Y) = \sum_{k=1}^n (-1)^{k-1} y_k h_k(Y)\frac{\partial}{\partial y_k}.\]

\no Let $c=(c_1,\dots,c_n)\in \ccc^n$ be a singular point of $\eta'$. By definition, for $k=1,\dots,n$, we have the equalities $c_kh_k(c) = 0$. As $\eta$ is NND, the second statement in \cref{Reduction theorem} guarantees that the $h_k$'s cannot all vanish at $c$. There exists $\ell$ such that $h_\ell(c)\ne 0$, so $c_\ell = 0$. We remark that $(-1)^{\ell -1}h_\ell(c)$ is a nonzero eigenvalue of the matrix of the linear part of $\chi$ at $c$. Indeed, its $\ell$-th column is


\begin{equation*}
\begin{pmatrix}
         (-1)^{\ell-1}c_\ell\partial_{y_1}h_\ell(c) \\
         \vdots \\
         (-1)^{\ell-1}c_\ell\partial_{y_{\ell -1}}h_\ell(c) \\
         (-1)^{\ell-1}(h_\ell(c)+c_\ell\partial_{y_\ell}h_\ell(c)) \\
         (-1)^{\ell-1}c_\ell\partial_{y_{\ell+1}}h_\ell(c) \\
         \vdots \\

        (-1)^{\ell-1}c_\ell\partial_{y_{n}}h_\ell(c) \\
\end{pmatrix} =
\begin{pmatrix}
         0 \\
         \vdots \\
         0 \\
         (-1)^{\ell-1}h_\ell(c) \\
         0 \\
         \vdots \\

        0 \\
\end{pmatrix}
.
\end{equation*}

\medskip

\no We conclude that the singular point $c$ is nonnilpotent.

\bigskip

\no Consider the case $|A| = n-1$. Without loss of generality, we can suppose $A=\{2,\dots,n\}$. According to the first statement in \cref{Reduction theorem}, for any subset $K\subset \{1,\dots,n\}$ with $n-1$ elements, $Y_{K\setminus A}$ divides $\overline{f}_K$. In our case, this means that for any $K\ne \{2,\dots,n\}$, $y_1$ divides $\overline{f}_K$. 

\medskip

\no The strict transform $\eta'$ of $\eta$ by $\pi$ is given by:

\begin{gather*}
    \eta'(Y) = y_2\dots y_n\left( \sum_{k = 1}^n  \overline{f}_{J_k}(Y) \frac{dY_{J_k}}{Y_{J_k}}\right) \\
    \eta'(Y) = \overline{f}_{J_1}(Y)+\sum_{k=2}^n y_k\frac{\overline{f}_{J_k}(Y)}{y_1}dY_{J_k} \\
    \eta'(Y) = h_1(Y)dY_{J_1}+\sum_{k=2}^ny_kh_k(Y)dY_{J_k}.
\end{gather*}

where $h_1 = \overline{f}_{J_1}$ and for any $k=2,\dots,n$, we set $h_k = \frac{\overline{f}_{J_k}}{y_1}$. Let $\chi$ the vector field defined by:

\[\chi(Y) = h_1(Y)\frac{\partial}{\partial y_1}+\sum_{k=2}^n (-1)^{k-1} y_k h_k(Y)\frac{\partial}{\partial y_k}.\]

\no Let $c=(c_1,\dots,c_n)\in \ccc^n$ be a singular point of $\eta'$. This means that $h_1(c) = 0$ and for $k=2,\dots,n$, $c_kh_k(c) = 0$. At least one coefficient $h_k$ with $k\ge 2$ does not vanish at $c$. Indeed, if for any $k=2,\dots,n$, $h_k(c)=0$, as $y_1h_k = \overline{f}_{J_k}$, we would have $\overline{f}_{J_k}(c) = 0$ for every $k=2,\dots,n$ and for $k=1$ also because $h_1=\overline{f}_{J_1}$. This is impossible according to \cref{Reduction theorem} since $\eta$ is NND.

\medskip

\no There exists $\ell\in \{2,\dots,n\}$ with $h_\ell(c)\ne 0$. As $c_\ell h_\ell(c)=0$, we have $c_\ell = 0$ and by the same line of reasoning as in the case $|A| = n$, we obtain that $(-1)^{\ell - 1} h_\ell(c)$ is a nonzero eigenvalue of the matrix of the linear part of $\chi$.

\end{proof}

\medskip

\bibliographystyle{halpha-abbrv.bst}
\bibliography{sample.bib}

\end{document}